\documentclass[11pt]{article}
\usepackage{amsmath,amscd,amsfonts,amssymb,amsthm,enumerate,
textcomp,xr,makeidx,stmaryrd,microtype,a4wide,indentfirst}
\usepackage[mathscr]{eucal}

\usepackage[english]{babel}
\usepackage[all,2cell]{xy}
\usepackage[colorlinks,plainpages,urlcolor=blue,linkcolor=reference,citecolor=citation]{hyperref}
\usepackage[abbrev,lite,alphabetic]{amsrefs}
\usepackage{enumitem}

\usepackage{color}
\definecolor{reference}{rgb}{.20,.60,.22}
\definecolor{citation}{rgb}{0,.40,.80}

\usepackage[left=2cm, top=2cm, right=2cm, bottom=2cm]{geometry}

\UseTwocells
\xyoption{2cell}{\UseTwocells}
\hypersetup{linktocpage}

\makeatletter
\DeclareRobustCommand{\em}{%
  \@nomath\em \if b\expandafter\@car\f@series\@nil
  \normalfont \else \bfseries \fi}
\makeatother

\theoremstyle{definition}
\newtheorem{Theor}{Theorem}[subsection]
\newtheorem{Lemma}[Theor]{Lemma}
\newtheorem{Prop}[Theor]{Proposition}
\newtheorem{Rem}[Theor]{Remark}
\newtheorem*{Rem*}{Remark}

\newtheorem{Def}[Theor]{Definition}
\newtheorem*{Conv*}{Conventions}
\newtheorem{Conv}[Theor]{Conventions}

\newtheorem{Cor}[Theor]{Corollary}

\newtheorem{Ex}[Theor]{Example}
\newtheorem{Exs}[Theor]{Example}
\newtheorem*{Plan}{Plan of the proof}
\newtheorem*{Ackn}{Acknowledgments}

\newcommand{\Arr}{\operatorname{\sf Arr}}
\newcommand{\Tr}{\operatorname{\sf Tr}}
\newcommand{\coev}{\operatorname{\sf coev}}
\newcommand{{\Vect}}{\operatorname{\sf Vect}}
\newcommand{\ch}{\operatorname{\sf ch}}

\newcommand{\Hom}{\operatorname{\sf Hom}}

\newcommand{\Funct}{\operatorname{\sf Funct}}

\newcommand{\Mmod}{\operatorname{\sf Mod}}
\newcommand{\Cat}{\operatorname{\sf Cat}}
\newcommand{\QCoh}{\operatorname{\sf QCoh}}

\newcommand{\st}{\operatorname{\sf st}}

\newcommand{\ev}{\operatorname{\sf ev}}

\newcommand{\Sch}{\operatorname{\sf Sch}}

\newcommand{\Spec}{\operatorname{\sf Spec}}

\newcommand{\ii}{\operatorname{(\infty,1)}}

\newcommand{\mscr}{\mathscr}

\newcommand{\mcal}{\mathcal}

\newcommand{\Id}{\operatorname{\sf Id}}

\newcommand{\Ind}{\operatorname{\sf Ind}}

\newcommand{\CAlg}{\operatorname{\sf CAlg}}

\newcommand{\op}{\operatorname{\sf op}}

\newcommand{\Ll}{\operatorname{\sf L}}

\newcommand{\fcat}{\mathscr}

\definecolor{note_color}{rgb}{0.0,0.9,0.0}

\renewcommand{\phi}{\varphi}
\renewcommand{\epsilon}{\varepsilon}
\newcommand{\lra}{\longrightarrow}

\DeclareMathOperator{\GL}{GL}

\DeclareMathOperator{\tr}{tr}

\DeclareMathOperator{\Perf}{Perf}

\DeclareMathOperator{\Fun}{Fun}

\newcommand{\indlim}{\lim\limits_{\lra}}

\mathchardef\mdef="2D

\DeclareMathOperator{\dual}{fd}

\begin{document}
\title{{\bfseries Categorical proof of Holomorphic Atiyah-Bott formula}}
\date{}

\author{Grigory~Kondyrev, Artem~Prikhodko}



\maketitle
\begin{abstract}
\noindent Given a $2$-commutative diagram
$$\xymatrix{
X \ar[d]_-{\varphi} \ar[r]^-{F_X} & X \ar[d]^-{\varphi}
\\
Y \ar[r]_-{F_Y} & Y
}$$
in a symmetric monoidal $(\infty,2)$-category $\mscr{E}$ where $X,Y \in \mscr{E}$ are dualizable objects and $\varphi$ admits a right adjoint we construct a natural morphism $\xymatrix{\Tr_{\mscr{E}}(F_X) \ar[r] & \Tr_{\mscr{E}}(F_Y)}$ between the traces of $F_X$ and $F_Y$ respectively. We then apply this formalism to the case when $\mscr{E}$ is the $(\infty,2)$-category of $k$-linear presentable categories which in combination of various calculations in the setting of derived algebraic geometry gives a categorical proof of the classical Atiyah-Bott formula (also known as the Holomorphic Lefschetz fixed point formula).
\end{abstract}
\tableofcontents

\section*{Introduction}
The well-known Lefschetz fixed point theorem \cite[Formula 71.1]{Lefschetz_fixed_points} states that for a compact manifold $M$ and an endomorphism $\xymatrix{M \ar[r]^-{f} & M}$ with isolated fixed points there is an equality
\begin{align}\label{formula:best_of_the_best}
L(f):=\sum_{i=0}^{2\dim M} (-1)^i\tr(f^*_{|H^i(X, \mathbb Q)}) = \sum_{x=f(x)} \deg_x(1-f)
\end{align}
The formula $(\ref{formula:best_of_the_best})$ made huge impact on algebraic geometry in $20$th century leading Grothendieck and co-authors to development of \'etale cohomology theory and their spectacular proof of Weil's conjectures. Not only the fixed point theorem admits various generalizations, it can also be stated in different contexts than in the original work of Lefschetz. For example, if a fixed point $x$ is simple, the index $\deg_x(1-f)$ admits differential-geometric description, namely there is an equality $\deg_x(1-f) = \pm \det(1-d_xf)$. This observation has a vast generalization due to Atiyah and Bott \cite{AtiyahBott_originalI}, \cite{AtiyahBott_originalII}: in the case when all the fixed points of $f$ are simple, given an elliptic complex $E$ on $M$ and a bundle map $\xymatrix{f^{-1}E \ar[r]^-{b} & E}$ there is an equality
$$L(E, b) := \sum_i (-1)^i \tr (H^i(b)_{|H^i(M, E)}) = \sum_{x=f(x)} \mu_x$$
where $\mu_x$ are some explicit infinitesimal invariant of $E$ and $f$ at $x$ (see below for a more concrete statement). The original Lefschetz formula (\ref{formula:best_of_the_best}) can be recovered by taking $E$ to be the de Rham complex $\Omega_{M, dR}^\bullet$ with its canonical equivariant structure.

In this work we are concerned with the algebro-geometric version of the Atiyah-Bott formula. Namely let $k$ be an algebraically closed field and $X$ be a smooth proper variety over $k$ together with an endomorphism $\xymatrix{X \ar[r]^-{f} & X}$ such that its graph $\xymatrix{X \ar[r]^-{\Gamma_f} & X \times X}$ intersects the diagonal $\xymatrix{X \ar[r]^-{\Delta} & X \times X}$ transversally so that the fixed point scheme $X^f$ is a disjoint union of finitely many (simple) points. Let us call a quasi-coherent sheaf $E$ on $X$ \emph{lax-equivariant} if there is a fixed morphism $\xymatrix{f^*E \ar[r]^-{b} & E}$ (prefix "lax" corresponds to the fact that $b$ is not required to be an equivalence). In this setting Atiyah-Bott formula (also known as holomorphic Lefschetz fixed point formula) says that for a dualizable lax-equivariant sheaf $E$ there is an equality of two numbers
$$
\Ll(E, b) = \sum_{x=f(x)} \frac{\Tr_{k}(E_x \simeq E_{f(x)} \stackrel{b_x}{\longrightarrow} E_x)}{\det(1-d_xf)}
$$
where $\xymatrix{T_{X,x} \ar[r]^-{d_x f } & T_{X,x}}$ is the differential of $f$ viewed as a map from the tangent space at a point $x \in X$ to itself and $\Ll(E,b) \in k$ is the Lefschetz number 
$$\xymatrix{
\Ll(E, b):=\Tr_{k} \Big(\Gamma(X, E) \ar[r] & \Gamma(X,f_*f^*E) \simeq \Gamma(X, f^*E) \ar[r]^-{\Gamma(b)} & \Gamma(X, E)\Big)
}$$
of $b$.

Since both of the parts of the Atiyah-Bott formula are expressed in terms of traces, it is naturally to try to derive the Atiyah-Bott formula using some general mechanism which allows to compare traces of different objects. Recent development of higher category theory and derived algebraic geometry provide an appropriate context to formulate such ideas more concretely. In this paper, we provide a categorical proof of the Atiyah-Bott formula.  Namely, we interpret both sides of the equality above as morphisms in the $(\infty,1)$-category of unbounded cochain complexes $\Vect_k$ from $k \in \Vect_k$ to itself. The desired equality then follows from the naturality of a certain construction in the world of $(\infty,2)$-categories.

\paragraph*{Plan of the paper.} In the first section we introduce the main categorical tool. Namely, given a commutative up to a (not necessarily invertible) $2$-morphism diagram of the form
$$\xymatrix{
X \ar[dd]_-{\varphi} \ar[rr]^-{F_X} && X \ar[dd]^-{\varphi} \ar@2[ddll]_-{T}
\\
\\
Y \ar[rr]_-{F_Y} && Y 
}$$
 in a symmetric monoidal $(\infty,2)$-category $\mscr{E}$, where $X,Y \in \mscr E$ are dualizable and $\varphi$ admits a right adjoint, we construct a morphism of traces
$$\xymatrix{
\Tr_{\mscr{E}}(F_X) \ar[rr]^-{\Tr(\varphi,T)} && \Tr_{\mscr{E}}(F_Y)
}$$
which is compatible with vertical compositions up to homotopy (see proposition \ref{prop:functoriality_of_traces} for a precise statement).

In the second section of this paper we apply this formalism to the setting of derived algebraic geometry by considering the case $\mscr{E}=2\Cat_k$, the $(\infty,2)$-category of $k$-linear stable presentable categories and continuous functors. It is well known (see e.g. \cite{BFN}) that for a quasi-compact quasi-separated derived scheme $X$ the $\ii$-category of unbounded cochain complexes of quasi-coherent sheaves $\QCoh(X)$ on $X$ being compactly generated is a dualizable object in $2\Cat_k$ (see proposition \ref{self_dual_qcoh} for more details), so we can apply the machinery of traces. Namely, given an endomorphism $\xymatrix{X \ar[r]^-{f} & X}$ of a derived scheme $X$ the functor $f_*$ induces an endomorphism $\xymatrix{\QCoh(X) \ar[r]^-{f_*} & \QCoh(X)}$ and we calculate (\ref{prop:tr_of_f}) that the corresponding trace is simply
$$\Tr_{2\Cat_k}(f_*) \simeq \Gamma(X^f, \mathcal O_{X^f}) \in \Hom_{2\Cat_k}(\Vect_k,\Vect_k) \simeq \Vect_k$$
where $X^f$ is the derived fixed point scheme (see definition \ref{def:derived_intersection_stack}).  Now a lax-equivariant sheaf $E \in \QCoh(X)$ as in the setting of the Atiyah-Bott formula allows us to construct a diagram
$$\xymatrix{
{\Vect_k} \ar[d]_-{E} \ar[rr]^-{\Id_{{\Vect_k}}} & & {\Vect_k} \ar[d]^-{E} \ar@2[dll]_-{T}
\\
\QCoh(X) \ar[rr]_-{f_*} && \QCoh(X) 
}$$
where the $2$-morphism $T$ corresponds to the morphism $\xymatrix{f^*E \ar[r]^-{b} & E}$. As $\Tr_{2\Cat_k} (\Id_{\Vect_k}) \simeq k$, the induced map of traces 
$$\xymatrix{
k \simeq \Tr_{2\Cat_k} (\Id_{\Vect_k}) \ar[rr]^-{\Tr(\varphi,T)} && \Tr_{2\Cat_k}(f_*) \simeq \Gamma(X^f, \mathcal O_{X^f})
}$$
is just a choice of an element in $\Gamma(X^f, \mathcal O_{X^f})$. The main computation in the second section is another characterization of this element: namely, if we denote by $\xymatrix{X^f \ar[r]^-{i} & X}$ the inclusion of the derived fixed points scheme, then proposition \ref{prop:chern_in_ag} establishes an equality
$$\Tr(\varphi,T) = \Tr_{\QCoh(X^f)}\left(\xymatrix{i^* E \simeq i^* f^* E \ar[r]^-{i^*(b)} & i^* E}\right)$$
which is extremely useful in further calculations.

In the last section we apply the above categorical machinery to the particular case of the Atiyah-Bott formula. Considering the diagram
$$\xymatrix{
{\Vect_k} \ar[d]_-{E} \ar[rr]^-{\Id_{{\Vect_k}}} & & {\Vect_k} \ar[d]^-{E} \ar@2[dll]_-{T}
\\
\QCoh(X) \ar[d]_{\Gamma} \ar[rr]_-{f_*} && \QCoh(X)  \ar[d]^-{\Gamma} 
\\
{\Vect_k} \ar[rr]_-{\Id_{{\Vect_k}}} && {\Vect_k}
}$$
we obtain a commutative triangle
$$\xymatrix{
\Tr_{2\Cat_k}(\Id_{{\Vect_k}}) \ar[rr]^-{\Tr(E,T)} \ar[drr]_-{\Tr(\Gamma(X,E),\Id_{\Gamma(X,E)} \circ T) \indent \indent} && \Tr_{2\Cat_k}(f_*) \ar[d]^-{\Tr(\Gamma, \Id_{\Gamma})} .
\\
&& \Tr_{2\Cat_k}(\Id_{{\Vect_k}})
}$$
in $\Vect_k$. Since $\Tr_{2\Cat_k}(\Id_{{\Vect_k}}) \simeq k$ this gives an equality of two {\bfseries numbers}. It is then a combination of formal verifications and calculations from section $2$ that the morphisms $\Tr(\Gamma(X,E),\Id_{\Gamma(X,E)})$ and $\Tr(\Gamma,\Id_{\Gamma}) \circ \Tr(E,T)$ are precisely the left-hand and the right-hand sides of the Atiyah-Bott formula respectively.

\begin{Rem*}
After writing the paper we were pointed that results of a similar manner were obtained in the work \cite{BN_ntr}. The main difference of the present work is a useful description of the chern character in geometric terms and a more concrete description of the functional on the derived stack of fixed points which allows to identify completely the categorical version of the statement with the classical Atiyah-Bott formula.
\end{Rem*}

\medskip

\begin{Ackn} We would like to thank Dennis Gaitsgory for suggesting the problem and numerous helpful discussions. We also want to thank anonymous referee for his thorough review and many useful suggestions.

The first author was supported by the Russian Academic Excellence Project '5-100'. The second author is partially supported by Laboratory of Mirror Symmetry NRUHSE, RF Government grant, ag. № 14.641.31.0001.
\end{Ackn}

\medskip
\begin{Conv*}\
\\
1) All the categories we work with are assumed to be $\ii$-categories. For an $\ii$-category $\mscr{C}$ we will denote by $(\mscr{C})^{\simeq}$ the underlying $\infty$-groupoid of $\mscr{C}$ obtained by discarding all the non-invertible morphisms from $\mscr{C}$. Analogously for an $(\infty,2)$-category $\mathscr E$ we will denote by $\mathscr E^{1 \mdef \sf cat}$ the maximal $\ii$-subcategory of $\mathscr E$ obtained by discarding all the non-invertible $2$-morphisms. For a symmetric monoidal category $\mathscr C$ we will denote the full subcategory of dualizable objects by $\mathscr C^{\dual}$.
\\
\\
2) We will denote by $\mcal{S}$ the symmetric monoidal $\ii$-category of spaces. For a field $k$ we will denote by ${\Vect_k}$ the stable symmetric monoidal $\ii$-category of unbounded cochain complexes over $k$ up to quasi-isomorphism with the canonical $(\infty,1)$-enhancement. We will also denote by $\Vect_k^{\heartsuit}$ the ordinary category of $k$-vector spaces considered as an $\ii$-category.
\\
\\
3) We will denote by $\Pr^{\Ll}_{\infty}$ the $\ii$-category of presentable $\ii$-categories and continuous functors with a symmetric monoidal structure from \cite[Proposition 4.8.1.14.]{HA}. Similarly, we will denote by $\Pr^{\Ll,\st}_{\infty}$ the $\ii$-category of stable presentable $\ii$-categories and continuous functors considered as a symmetric monoidal $\ii$-category with the monoidal structure inherited from $\Pr^{\Ll}_{\infty}$.
\\
\\
4) Notice that $\Vect_k$ is a commutative algebra object in $\Pr^{\Ll,\st}_{\infty}$. By \cite[Theorem 4.5.2.1.]{HA} it follows that the presentable stable $\ii$-category of $k$-linear presentable $\ii$-categories and $k$-linear functors $\Cat_k:=\Mmod_{\Vect_k}(\Pr^{\Ll,\st}_{\infty})$ admits natural symmetric monoidal structure. We will also denote by $2\Cat_k$, the symmetric monoidal $(\infty,2)$-category of $k$-linear presentable $\ii$-categories and continuous $k$-linear functors so that the underlying $\ii$-category of $2\Cat_k$ is precisely $\Cat_k$.
\end{Conv*}

\section{Dualizable objects and traces}

\subsection{Traces in symmetric monoidal $\ii$-categories}
We start with the following well-known

\begin{Def}\label{def:trace}
Let $\xymatrix{X \ar[r]^-{f} & X}$ be a morphism in a symmetric monoidal $\ii$-category $\mscr{C}$ with $X \in \mscr{C}$ being dualizable. Define then the {\bfseries trace of $f$} denoted by $\Tr_{\mscr{C}}(f) \in\ \Hom_{\mscr{C}}(I,I)$ as the composite
$$\xymatrix{
I \ar[rr]^-{\coev} && X \otimes X^\vee \ar[rr]^-{f \otimes \Id_{X^\vee}} && X \otimes X^\vee \ar[rr]^-{{\sf Twist}}_-{\sim} && X^\vee \otimes X \ar[rr]^-{\ev} && I
}$$
where $I$ is the monoidal unit in $\mscr{C}$.
In particular, for any dualizable object $X \in \mscr{C}$ we define the {\bfseries dimension of $X$} denoted by $\mcal{X}_{\mscr{C}}(X) \in \Hom_{\mscr{C}}(I,I)$ simply as the trace of the identity map $\mcal{X}_{\mscr{C}}(X):=\Tr_{\mscr{C}}(\Id_X)$.
\end{Def}

\begin{Rem}
Notice that trace is cyclic: given two morphisms $\xymatrix{X \ar[r]^-{f} & Y}$  and $\xymatrix{Y \ar[r]^-{g} & X}$ where $X$ and $Y$ are both dualizable there is a canonical equivalence $\Tr_{\mscr{C}}(f \circ g) \simeq \Tr_{\mscr{C}}(g \circ f)$
\end{Rem}

\begin{Exs}\ 
\\
1) Notice that an object $V \in \Vect_k$ is dualizable if an only if it has finite-dimensional cohomology spaces nonzero only in finitely many degrees. If $V \in \Vect_k$ is dualizable and $f \in \Hom_{\Vect_k}(V,V)$ is some morphism then the trace $\Tr_{\Vect_k}(f) \in \Hom_{\Vect_k}(k,k) \simeq k$ is the alternating sum of the ranks of the map on the cohomology spaces of $V$ induced by $f$. In particular, in case $f=\Id_V$ we see that $\Tr_{\Vect_k}(\Id_V)=\mcal{X}_{\Vect_k}(V) \in k$ is simply the Euler characteristic of $V$.
\\
\\
2) Let $\mscr{D} \in \Cat_k$ be dualizable. Recall that the monoidal unit of $\Cat_k$ is the $\ii$-category $\Vect_k \in \Cat_k$ and, moreover, we have 
$$\Hom_{\Cat_k}(\Vect_k,\Vect_k)=\Hom_{\Mmod_{\Vect_k}(\Pr^{\Ll,\st}_{\infty})}(\Vect_k,\Vect_k) \simeq (\Vect_k)^{\simeq}.
$$
The cochain complex $\mcal{X}_{\Cat_k}(\mscr{D}) \in (\Vect_k)^{\simeq}$ is frequently called {\bfseries Hochschild homology of $\mscr{D}$}.
\end{Exs}

\subsection{Traces in symmetric monoidal $(\infty,2)$-categories}
Let $\mscr{E}$ be a symmetric monoidal $(\infty,2)$-category (that is, a commutative algebra object in the $\ii$-category of $(\infty,2)$-categories, see \cite[Chapter V.3, 1.4.1.]{GaitsRoz}) and $X,Y \in \mscr{E}$ be two dualizable objects in $\mscr{E}$. Suppose we are given a (not necessary commutative) diagram
$$
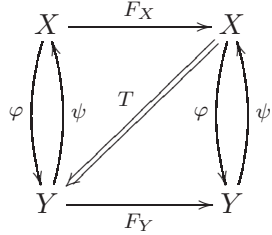
\xymatrix{
X \ar@/_/[dd]_-{\varphi} \ar[rr]^-{F_X} && X \ar@/_/[dd]_-{\varphi} \ar@2[ddll]_-{T}
\\
\\
Y \ar@/_/[uu]_-{\psi} \ar[rr]_-{F_Y} && Y \ar@/_/[uu]_-{\psi}
}$$
in $\mscr{E}$, where $\varphi$ is left adjoint to $\psi$ and
$$\xymatrix{ \varphi \circ F_X \ar[rr]^-{T} && F_Y \circ \varphi}$$
is a $2$-morphism in $\mscr{E}$. We argue that then there is a natural morphism 
$$\xymatrix{
\Tr_{\mscr{E}}(F_X) \ar[rr]^-{\Tr(\varphi,T)} && \Tr_{\mscr{E}}(F_Y)
}$$
in the $\ii$-category $\Hom_{\mscr{E}}(I,I)$ which is very useful in various applications. Our plan for this section is to define the morphism $\Tr(\varphi,T)$ as the trace in a certain $\ii$-category of arrows. We begin with the following

\begin{Def}\label{def:category_of_arrows}
Let $\mscr{E}$ be an $(\infty,2)$-category. We define {\bfseries an $\ii$-category of arrows $\Arr(\mscr{E})$} in $\mscr{E}$ as the maximal $\ii$-subcategory of the $(\infty,2)$-category of unital left-lax functors from $[1]$ to $\mathscr E$
$$\Arr(\mathscr E) := \Fun([1], \mathscr E)^{1 \mdef \sf cat}_{{\sf left \mdef lax}}$$
We refer the reader to \cite[Chapter  A.1, 3.1]{GaitsRoz} for a discussion of lax-functors between $(\infty,2)$-categories.
\end{Def}
\medskip 

In particular, we see that by definition the space of objects of $\Arr(\mathscr E)$ is the space of arrows in $\mscr{E}$, that is, $(\Arr(\mscr{E}))^{\simeq}:=(\Funct(\Delta^1,\mscr{E}))^{\simeq}$. Moreover, given two objects $\xymatrix{X \ar[r]^-{f} & Y}$ and $\xymatrix{Z \ar[r]^-{g} & W}$ in $\Arr(\mscr{E})$ the space of morphisms between them is the space of diagrams
$$
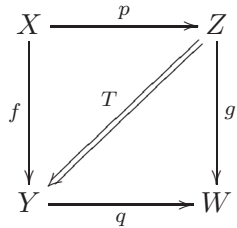
\xymatrix{
X \ar[dd]_-{f} \ar[rr]^-{p} && Z \ar[dd]^-{g} \ar@2[ddll]_-{T}
\\
\\
Y \ar[rr]_-{q} && W
}$$
where $\xymatrix{g \circ p \ar[r]^-{T} & q \circ f}$ is a chosen $2$-morphism in $\mscr{E}$.

\begin{Ex}\label{Ex:end_in_arr}
By the construction of $\Arr(\mscr{E})$ we see that for an object $X \in \mscr{E}$ the space
$$\xymatrix{\Hom_{\Arr(\mscr{E})}(X \ar[r]^-{\Id_X} & X,X \ar[r]^-{\Id_X} & X)}$$ consists of triples $(f,g,T)$, where $f,g \in \Hom_{\mscr{E}}(X,X)$ and $T \in \Hom_{\Hom_{\mscr{E}}(X,X)}(f,g)$.
\end{Ex}

Note that the canonical inclusions $\xymatrix{[0] \ar@<-.5ex>[r] \ar@<.5ex>[r] & [1]}$ induce two functor $\xymatrix{\Arr(\mathscr E) \ar@<-.5ex>[r]_-{t_{\mathscr E}} \ar@<.5ex>[r]^-{s_{\mathscr E}} & \mathscr E^{\sf 1 \mdef cat}}$ which on the level of objects send an arrow $\xymatrix{X \ar[r] & Y}$ to $X$ and $Y$ respectively. This is important due to the following

\begin{Prop}\label{rem:dualizable_in_arrows}
Let $\mscr{E}$ be a symmetric monoidal $(\infty,2)$-category. Then
\begin{enumerate}
\item The $\ii$-category $\Arr(\mscr{E})$ admits symmetric monoidal structure such that functors $\xymatrix{\Arr(\mathscr E) \ar@<-.5ex>[r]_-{t_{\mathscr E}} \ar@<.5ex>[r]^-{s_{\mathscr E}} & \mathscr E^{\sf 1 \mdef cat}}$ are both symmetric monoidal.

\item An object $\xymatrix{X \ar[r]^-{\varphi} & Y \in \Arr(\mscr{E})}$ is dualizable if and only if $X$ and $Y$ are dualizable objects of $\mscr{E}$ and the morphism $\varphi$ admits a right adjoint $\psi$. 
\end{enumerate}

\begin{proof}
\begin{enumerate}[leftmargin=0pt,itemindent=*]
\item Recall first that by definition \cite[Chapter  A.1, 3.1.3.]{GaitsRoz} for any $(\infty,2)$-category $\mathscr F$ we have $\Fun(\mathscr F, \mathscr E)_{\sf left \mdef lax} = \Fun(\mathscr F^{\sf 2 \mdef op}, \mathscr E^{\sf 2\mdef op})_{\sf right \mdef lax}$.  Now since $\Fun(\mathscr F,-)_{\sf right \mdef lax}$ is right adjoint to the Gray product functor (see \cite[Chapter A.1, 3.2.7]{GaitsRoz}) it preserves limits. Since the functor $(-)^{\sf 2\mdef op}$ is an autoequivalence and the functor $(-)^{\sf 1\mdef cat}$ is right adjoint, the composite $\Arr(-) = \Fun([1], -^{\sf 2\mdef op})_{\sf right \mdef lax}^{\sf 1\mdef cat}$ also preserves limits and in particular it preserves finite products. It follows by \cite[Corollary 2.4.1.8]{HA} the functor $\Arr(-)$ can be promoted to a symmetric monoidal functor with respect to the cartesian monoidal structure. Now since by the assumption $\mathscr E$ is a commutative monoid in the $(\infty,1)$-category of $(\infty,2)$-categories and symmetric monoidal functors preserve commutative monoids, we obtain that $\Arr(\mathscr E)$ is a commutative monoid in $\Cat_{\infty}$, i.e. a symmetric monoidal category.

For the second part, notice that the canonical inclusions $\xymatrix{[0] \ar@<-.5ex>[r] \ar@<.5ex>[r] & [1]}$ induce natural transformations of functors $\xymatrix{\Arr(-) \ar@<-.5ex>[r]_-{t} \ar@<.5ex>[r]^-{s} & (-)^{\sf 1\mdef cat}}$. Since by \cite[Corollary 2.4.1.8]{HA} the space of symmetric monoidal natural transformations between symmetric monoidal functors with respect to cartesian monoidal structure is equivalent to the space of ordinary natural transformations, we see that $s$ and $t$ are symmetric monoidal transformations. In particular, it follows that the corresponding morphisms $\xymatrix{\Arr(\mathscr E) \ar@<-.5ex>[r]_-{t_{\mathscr E}} \ar@<.5ex>[r]^-{s_{\mathscr E}} & \mathscr E^{\sf 1 \mdef cat}}$ are maps of commutative monoids in $\Cat_{\infty}$, i.e. are symmetric monoidal functors.

\item Note first that if $X$ and $Y$ are dualizable objects, then given any morphism $\xymatrix{Y \ar[r]^-{\psi} & X}$ the morphism
$$\xymatrix{
(\Id_Y \otimes \psi^{\vee}) \circ \ev_Y: Y \otimes X^{\vee} \ar[r]^-{\coev_Y} & Y \otimes X^{\vee} \otimes Y \otimes Y^{\vee} \ar[r]^-{\psi} & Y \otimes X^{\vee} \otimes X \otimes Y^{\vee} \ar[r]^-{\ev_X} & Y \otimes Y^{\vee} \ar[r]^-{\ev_Y} & I
}$$
can be rewritten as
$$\xymatrix{
Y \otimes X^{\vee} \ar[r]^-{\coev_Y} & Y \otimes X^{\vee} \otimes Y \otimes Y^{\vee} \simeq Y \otimes Y^{\vee} \otimes X^{\vee} \otimes Y \ar[r]^-{\ev_Y} & X^{\vee} \otimes Y \ar[r]^-{\psi} & X^{\vee} \otimes X \ar[r]^-{\ev_X} & I
}$$
and so by triangle identities we see that the square
$$\xymatrix{
Y \otimes X^{\vee} \ar[rr]^-{\psi \otimes \Id_{X^{\vee}}} \ar[d]_-{\Id_Y \otimes \psi^{\vee}} && X \otimes X^{\vee} \ar[d]^-{\ev_X}
\\
Y \otimes Y^{\vee} \ar[rr]_-{\ev_Y} && I
}$$
commutes. In particular, for any morphism $\xymatrix{X \ar[r]^-{\varphi} & Y}$ we get an equivalence
$$
\ev_X \circ ((\psi \circ \varphi) \otimes \Id_{X^{\vee}}) \simeq \ev_X \circ(\psi \otimes \Id_{X^{\vee}}) \circ (\varphi \otimes \Id_{X^{\vee}}) \simeq \ev_Y \circ (\Id_Y \otimes \psi^{\vee}) \circ (\varphi \otimes \Id_{X^{\vee}}) \simeq \ev_Y \circ (\varphi \otimes \psi^\vee)
$$
in $\Hom_{\mscr{E}}(X \otimes X^{\vee},I)$. In a similar manner there is an equivalence
$$
(\varphi \otimes \psi^\vee) \circ \coev_X \simeq  ((\varphi \circ \psi) \otimes \Id_{Y^{\vee}})
$$
in $\Hom_{\mscr{E}}(I,Y \otimes Y^{\vee})$. 
\\
We now argue that if $X$ and $Y$ are dualizable and $\varphi$ is left adjoint to $\psi$, then dual object to the object $\xymatrix{X \ar[r]^-{\varphi} & Y \in \Arr(\mscr{E})}$ is simply $\xymatrix{X^\vee \ar[r]^-{\psi^\vee} & Y^\vee \in \Arr(\mscr{E})}$: define evaluation morphism as
$$\xymatrix{
X \otimes X^\vee \ar[dd]_-{\varphi \otimes \psi^\vee} \ar[rr]^-{\ev_X} && I \ar[dd]^-{\Id_I} \ar@2[ddll]_-{T_1}
\\
\\
Y \otimes Y^\vee  \ar[rr]_-{\ev_Y} && I
}$$
where $T_1$ is 
$$\xymatrix{
\ev_X \ar[rr] && \ev_X \circ ((\psi \circ \varphi) \otimes \Id_{X^{\vee}}) \simeq \ev_Y \circ (\varphi \otimes \psi^\vee)
}$$
induced by the unit of the adjunction between $\varphi$ and $\psi$ and the coevaluation morphism as
$$\xymatrix{
I \ar[dd]_-{\Id_I} \ar[rr]^-{\coev_X} && X \otimes X^\vee  \ar[dd]^-{\varphi \otimes \psi^\vee} \ar@2[ddll]_-{T_2}
\\
\\
I \ar[rr]_-{\coev_Y} && Y \otimes Y^\vee
}$$
where $T_2$ is $$\xymatrix{
(\varphi \otimes \psi^\vee) \circ \coev_X \simeq ((\varphi \circ \psi) \otimes \Id_{Y^{\vee}}) \circ \coev_Y \ar[rr] && \coev_Y
}$$
induced by the counit of the adjunction between $\varphi$ and $\psi$. It is straightforward to check that such morphisms satisfy triangle identities.
\\
Conversely, if an object $\xymatrix{X \ar[r]^-{\varphi} & Y \in \Arr(\mscr{E})}$ is dualizable then since the monoidal structure on $\Arr(\mscr{E})$ is defined pointwise its dual has to have a form $\xymatrix{X^\vee \ar[r]^-{\psi^\vee} & Y^\vee}$, where $\psi^{\vee}$ is some morphism. Now to see that $\psi$ is right adjoint to $\varphi$ we notice that the evaluation diagram gives a morphism $\xymatrix{\Id_X \ar[r]^-{T_1} & \psi \circ \varphi}$ and the coevaluation diagram gives a morphism $\xymatrix{\varphi \circ \psi \ar[r]^-{T_2} & \Id_Y}$. The classical conditions on evaluation and coevaluation morphisms are then precisely the triangle identities on $T_1$ and $T_2$.
\end{enumerate}
\end{proof}
\end{Prop}

\begin{Cor} \label{rem:dual_arrows_in2cat} We see that in the case when $\mscr{E}=2\Cat_k$ an arrow $\xymatrix{(\mscr{C} \ar[r]^-{F} & \mscr{D}) \in \Arr(2\Cat_k)}$ is a dualizable object iff both $\mscr{C},\mscr{D} \in \Cat_k$ are dualizable and the functor $F$ admits a right adjoint in $\Cat_k$, i.e. the functor $F$ admits a \textit{continuous} right adjoint as a plain functor.
\end{Cor}
 
 \begin{Ex} \label{rem:trace_concretely}
  Let $\mscr{E}$ be a symmetric monoidal $(\infty,2)$-category and $\xymatrix{X \ar[r]^-{\psi} & Y \in \Arr(\mscr{E})}$ be a dualizable object. Then for a morphism 
 $$\xymatrix{(F_X,F_Y,T) \in \Hom_{\Arr(\mscr{E})}(X \ar[r]^-{\varphi} & Y,X \ar[r]^-{\varphi} & Y)}$$ 
 in $\Arr(\mscr{E})$ which we imagine as a (not necessary commutative) diagram
$$\xymatrix{
X \ar[dd]_-{\varphi} \ar[rr]^-{F_X} && X \ar[dd]^-{\varphi} \ar@2[ddll]_-{T}
\\
\\
Y \ \ar[rr]_-{F_Y} && Y 
}$$
in $\mscr{E}$, where $\xymatrix{ \varphi \circ F_X \ar[rr]^-{T} && F_Y \circ \varphi}$ is a $2$-morphism the trace 
$$\xymatrix{\Tr_{\Arr(\mscr{E})}(F_X,F_Y,T) \in \Hom_{\Arr(\mscr{E})}(I \ar[r]^-{\Id_I} & I,I \ar[r]^-{\Id_I} & I)}$$
 is given by the big rectangle in the $2$-commutative diagram
$$\xymatrix{
I \ar[dd]_-{\Id_I} \ar[rr]^-{\coev_X} && X \otimes X^\vee \ar@2[ddll]  \ar[rr]^-{F_X \otimes \Id_{X^\vee}} \ar[dd]_-{\varphi \otimes \psi^\vee} && X \otimes X^\vee \ar[rr]^-{\ev_X} \ar[dd]^-{\varphi \otimes \psi^\vee} \ar@2[ddll] && I \ar[dd]^-{\Id_I} \ar@2[ddll] 
\\
\\
I \ar[rr]_-{\coev_Y} && Y \otimes Y^\vee \ar[rr]_-{F_Y \otimes \Id_{Y^\vee}} && Y \otimes Y^\vee \ar[rr]_-{\ev_Y} && I
}$$
where
\\
1) The $2$-morphism in the middle square is $T \otimes \psi^{\vee}$.
\\
2) The $2$-morphism in the left square is
$$\xymatrix{
(\varphi \otimes \psi^\vee) \circ \coev_X \simeq ((\varphi \circ \psi) \otimes \Id_{Y^{\vee}}) \circ \coev_Y \ar[rr] && \coev_Y
}$$
induced by the counit of the adjunction between $\varphi$ and $\psi$.
\\
3) The $2$-morphism in the right square is
$$\xymatrix{
\ev_X \ar[rr] && \ev_X \circ ((\psi \circ \varphi) \otimes \Id_{X^{\vee}}) \simeq \ev_Y \circ (\varphi \otimes \psi^\vee)
}$$
induced by the unit of the adjunction between $\varphi$ and $\psi$.
\\
Since the top row is simply $\Tr_{\mscr{E}}(F_X)$ and the bottom row is simply $\Tr_{\mscr{E}}(F_Y)$ we see that it makes sense to think of the trace $\Tr_{\Arr(\mscr{E})}(F_X,F_Y,T)$ as of a morphism from $\Tr_{\mscr{E}}(F_X)$ to $\Tr_{\mscr{E}}(F_Y)$ in the $\ii$-category $\Hom_{\mscr{E}}(I,I)$.
 \end{Ex}
 
 The example above motivates the following
 
\begin{Def}\label{def:map_of_traces}
  Let $\mscr{E}$ be a symmetric monoidal $(\infty,2)$-category and $\xymatrix{X \ar[r]^-{\varphi} & Y \in \Arr(\mscr{E})}$ be a dualizable object. Define then a {\bfseries morphism of traces $\Tr(\varphi,T)$ induced by $T$} 
$$\xymatrix{
\Tr_{\mscr{E}}(F_X) \ar[rr]^-{\Tr(\varphi,T)} &&  \Tr_{\mscr{E}}(F_Y) 
}$$
as the morphism from $\Tr_{\mscr{E}}(F_X)$ to $\Tr_{\mscr{E}}(F_Y)$ in the $\ii$-category $\Hom_{\mscr{E}}(I,I)$ defined by the trace $\Tr_{\Arr(\mscr{E})}(F_X,F_Y,T)$.
\end{Def}

In order to proceed further recall the following

\begin{Def}
Let $\mathscr C$ be an $\ii$-category. We will call an an object $X \in \mathscr C$ \emph{compact} if for any filtered diagram $\{Y_i\}_{i\in I}$ in $\mathscr C$ with colimit $Y$ the canonical map
$$\indlim \Hom_{\mathscr C}(X, Y_i) \to \Hom_{\mathscr C}(X, Y)$$
is an equivalence. Note that if $\mathscr C$ is stable one can equivalently ask $\Hom_{\mathscr C}(X, -)$ to commute with all colimits since $\Hom_{\fcat C}(X,-)$ always commutes with finite colimits and every colimit diagram is a filtered colimit of its finite subdiagrams.
\end{Def}

Here is an example relevant to this paper

\begin{Ex}
We claim that the $\ii$-category of morphisms $\xymatrix{\Vect_k \ar[r]^-{\varphi} & Y}$ in $2\Cat_k$ which admit a right adjoint $\psi$ in $2\Cat_k$ is equivalent to the full $\ii$-subcategory $Y^{\sf comp} \subseteq Y$ of $Y$ spanned by compact objects. Indeed, due to the equivalence
$$\Fun_{\Cat_k}(\Vect_k, Y) \simeq Y \quad \xymatrix{F \ar@{|->}[r]& F(k)}$$
it is enough to prove that $F(k)$ is compact if and only if $F$ admits a continuous right adjoint. But $\Hom_Y(F(k), -)$ is always a (not necessary continuous) right adjoint of $F$ and since $Y$ is stable $\Hom_Y(F(k), -)$ is continuous if and only if $F(k)$ is compact.
\end{Ex}

\begin{Def} \label{def:chern} 
Consider the case when $\mscr{E}=2\Cat_k$, $X=\Vect_k$ and $F_X=\Id_{\Vect_k}$ so that we have a diagram
$$\xymatrix{
\Vect_k \ar@/_/[dd]_-{\varphi} \ar[rr]^-{\Id_{\Vect_k}} && \Vect_k \ar@/_/[dd]_-{\varphi} \ar@2[ddll]_-{T}
\\
\\
Y \ar@/_/[uu]_-{\psi} \ar[rr]_-{F_Y} && Y \ar@/_/[uu]_-{\psi}.
}$$
Since by the previous example $\phi$ just classifies some compact object of $Y$ we see that the $2$-morphism $T$ simply corresponds to some morphism $t \in \Hom_{Y}(E,F_Y(E))$, where $E:=\varphi(k)$.
\\
In particular, we get an element in $\Tr_{2\Cat_k}(F_Y)$ which corresponds to the morphism
$$\xymatrix{
k \simeq \Tr_{2\Cat_k}(\Id_{{\Vect_k}}) \ar[rr]^-{\Tr(\varphi,T)} && \Tr_{2\Cat_k}(F_Y) \in \Hom_{2\Cat_k}(\Vect_k,\Vect_k) \simeq \Vect_k
}$$
called the {\bfseries Chern character of $E$} which will be further denoted by $\ch(E,t) \in \Tr_{2\Cat_k}(F_Y)$.
\end{Def}

\begin{Ex}\label{ex:chern_for_vect}
Consider the case when $\mscr{E}=2\Cat_k$, $X=Y=\Vect_k$ and $F_X=F_Y=\Id_{\Vect_k}$.  Set $V:=\varphi(k)$ so that $\varphi(-) \simeq V \otimes -$ and $\psi(-) \simeq V^\vee \otimes -$. We then have a diagram
$$\xymatrix{
\Vect_k \ar@/_/[dd]_-{V} \ar[rr]^-{\Id_{\Vect_k}} && \Vect_k  \ar@/_/[dd]_-{V} \ar@2[ddll]_-{T}
\\
\\
\Vect_k \ar@/_/[uu]_-{V^\vee} \ar[rr]_-{\Id_{\Vect_k}} && \Vect_k  \ar@/_/[uu]_-{V^\vee}
}$$
so that $T$ simply corresponds to some morphism $t \in \Hom_{\Vect_k}(V,V)$. Then directly by the definition \ref{def:map_of_traces} we get an equality
$$
\ch(V,t) = \Tr_{\Vect_k}(t)
$$
of two numbers.
\end{Ex}

\begin{Prop} \label{prop:functoriality_of_traces}
Let $\mscr{E}$ be a symmetric monoidal $(\infty,2)$-category and $X,Y,Z \in \mscr{E}$ be dualizable objects.  Suppose we are given a diagram 
$$\xymatrix{
X \ar@/_/[dd]_-{\varphi_1} \ar[rr]^-{F_X} && X \ar@/_/[dd]_-{\varphi_1} \ar@2[ddll]_-{T_1}
\\
\\
Y \ar@/_/[uu]_-{\psi_1} \ar@/_/[dd]_-{\varphi_2} \ar[rr]_-{F_Y} && Y \ar@/_/[uu]_-{\psi_1} \ar@2[ddll]_-{T_2} \ar@/_/[dd]_-{\varphi_2}
\\
\\
Z \ar@/_/[uu]_-{\psi_2}  \ar[rr]_-{F_Z} && Z \ar@/_/[uu]_-{\psi_2}
}$$
in $\mscr{E}$, where $\varphi_1$ is left adjoint to $\psi_1$, $\varphi_2$ is left adjoint to $\psi_2$ and
$$\xymatrix{ \varphi_1 \circ F_X \ar[rr]^-{T_1} && F_Y \circ \varphi_1}$$
$$\xymatrix{ \varphi_2 \circ F_Y \ar[rr]^-{T_2} && F_Z \circ \varphi_2}$$
are $2$-morphisms. Then there is an equivalence
$$
\Tr(\varphi_2 \circ \varphi_1, T_2 \circ_v T_1) \simeq \Tr(\varphi_2, T_2) \circ \Tr(\varphi_1,T_1)
$$
where $\circ_v$ is the vertical composition of $2$-morphisms.

\begin{proof}
Using the explicit description of the morphism of traces from \ref{rem:trace_concretely} the claim follows from the fact that given a diagram
$$\xymatrix{X  \rrtwocell^{f}_{f'}{_{\alpha}} && Y \rrtwocell^{g}_{g'}{_{\beta}} && Z }$$
in any $(\infty,2)$-category the diagram
$$\xymatrix{
g \circ f \ar[d]_-{\Id_g \circ \alpha} \ar[rr]^-{\beta \circ \Id_{f}} && g' \circ f \ar[d]^-{\Id_{g'} \circ \alpha}
\\
g \circ f' \ar[rr]_-{\beta \circ \Id_{f'}} && g' \circ f'
}$$
commutes.
\end{proof}
\end{Prop}

\section{Traces in algebraic geometry}
\begin{Conv}\
We will further work in the setting of derived algebraic geometry over some fixed field $k$. For a derived stack $X$ over a field $k$ we will denote the $k$-linear symmetric monoidal $\ii$-category of unbounded complexes of quasi-coherent sheaves on $X$ by $\QCoh(X) \in \CAlg(\Cat_k)$. Similarly, all the functors are assumed to be derived in the appropriate sense. We refer the reader to \cite{GaitsRoz} for an introduction to the basic concepts of derived algebraic geometry using the functor of points approach.
\end{Conv}

\subsection{Duality for Quasi-Coherent sheaves}
In this section we quickly review basics statements about quasi-coherent sheaves on perfect stacks (see \cite[Definition 3.2]{BFN} for the definition of a perfect stack). In particular, recall from \cite[Proposition 3.19]{BFN} that any quasi-compact derived scheme with affine diagonal is perfect. 

\begin{Prop}\label{self_dual_qcoh}
Let $X$ a perfect derived stack. Then $\QCoh(X) \in \Cat_k$ is self-dual.

\begin{proof}
Recall that by \cite[Chapter I.1, Proposition 7.3.2]{GaitsRoz} for a small stable category $\mathscr C_0$ we have an equivalence $\Ind(\mathscr C_0)^\vee \simeq \Ind(\mathscr C_0^{\op})$. Now since $X$ is perfect, we have $\QCoh(X)\simeq \Ind \Perf(X)$. Moreover, by \cite[Proposition 3.6]{BFN} the full subcategory $\QCoh(X)^{\dual}$ of dualizable objects in $\QCoh(X)$ coincides with $\Perf(X)$ so we get an equivalence $\xymatrix{\Perf(X)^{\op} \ar[r]^-{(-)^{\vee}}_-{\sim} & \Perf(X)}$ given by dualization and the result follows. 
\end{proof}
\end{Prop}

In what follows we will need a concrete description of self-duality of $\QCoh(X)$ for $X$ as above. Recall 

\begin{Theor}[{\cite[Theorem 1.2]{BFN}}]
For two perfect derived stacks $X, Y$ functor
$$\xymatrix{
\QCoh(X) \otimes \QCoh(Y) \ar[rr]^-{\boxtimes} && \QCoh(X \times Y) 
}$$ 
is an equivalence in $\Cat_k$.
\end{Theor}

In particular, by considering the diagram 
$$\xymatrix{
\ast && \ar[ll]_-{p} X \ar[rr]^-{\Delta} && X \times X.
}$$
we can define the evaluation map
$$\xymatrix{
\QCoh(X) \otimes \QCoh(X) \simeq \QCoh(X \times X) \ar[rr]^-{\sf ev_{\QCoh(X)}} && \Vect_k
}$$
simply as
$$
{\sf ev}:=p_* \circ \Delta^*
$$
and the coevaluation map
$$\xymatrix{
\Vect_k \ar[rr]^-{\sf coev_{\QCoh(X)}} && \QCoh(X \times X) \simeq \QCoh(X) \otimes \QCoh(X)
}$$
as
$$
{\sf coev}:=\Delta_* \circ p^*.
$$
One can check that the morphisms above realize $\QCoh(X)$ as a self-dual object. See \cite[Corollary 4.8]{BFN} for a bit more general statement.

\begin{Rem}
More generally, recall that by \cite[Chapter II.2, 5.3.2]{GaitsRoz} the quasi-coherent sheaves functor can be lifted to a symmetric monoidal functor
$$\xymatrix{
{\sf Corr}(\Sch)^{\sf all}_{{\sf all},{\sf all}} \ar[rr] && (2\Cat_k)^{2-\op}
}$$
where ${\sf Corr}(\Sch)^{\sf all}_{{\sf all},{\sf all}}$ is a symmetric monoidal $(\infty,2)$-category whose objects are derived schemes in the sense of  \cite[Chapter I.2, 3.1.1]{GaitsRoz}, morphisms from $X$ to $Y$ are spans
$$\xymatrix{
X & W \ar[r] \ar[l] & Y
}$$
with the composition given by pullbacks, $2$-morphisms are commutative diagrams 
$$\xymatrix{
& W_1 \ar[dr] \ar[dl] \ar[dd]^-{h}
\\
X && Y
\\
& W_2 \ar[ur] \ar[ul]
}$$
and the monoidal structure is given by cartesian product. Informally speaking, the extension of the quasi-coherent sheaves to the category of correspondences is given by mapping the span $\xymatrix{X & W \ar[r]^-{t} \ar[l]_-{s} & Y}$ to the morphism $\xymatrix{\QCoh(X) \ar[r]^-{s^*} & \QCoh(W) \ar[r]^-{t_*} & \QCoh(Y)}$. 
\\
In particular, the fact that $\QCoh(X) \in \Cat_k$ is self dual via the morphisms as in the discussion above can be derived from the fact that the spans
$$\xymatrix{
\ast && \ar[ll]_-{p} X \ar[rr]^-{\Delta} && X \times X
}$$
and
$$\xymatrix{
X \times X && \ar[ll]_-{\Delta} X \ar[rr]^-{p} && \ast
}$$
realize $X \in {\sf Corr}(\Sch)^{\sf all}_{{\sf all},{\sf all}}$ as a self-dual object. We refer to \cite[Chapter V.1, Chapter II.2]{GaitsRoz} for a thorough discussion of the category of correspondences. 
\end{Rem}

In particular, by the discussion above we get the following
\begin{Cor}
For two perfect derived stacks $X$ and $Y$ there is an equivalence
$$
\Hom_{\Cat_k}(\QCoh(X),\QCoh(Y)) \simeq \Hom_{\Cat_k}(\Vect_k, \QCoh(X) \otimes \QCoh(Y)) \simeq
$$
$$
\simeq (\QCoh(X) \otimes \QCoh(Y))^{\simeq} \simeq \QCoh(X \times Y)^{\simeq}.
$$
Concretely, for a sheaf $\mcal{K} \in \QCoh(X \times Y)$ the corresponding functor from $\QCoh(X)$ to $\QCoh(Y)$ is
$$
{q_2}_* (\mcal{K} \otimes (q_1^* -)) \in \Hom_{\Cat_k}(\QCoh(X),\QCoh(Y))$$
where 
$$\xymatrix{
X && X \times Y \ar[ll]_-{q_1} \ar[rr]^-{q_2} && Y
}$$
are the projection maps. The sheaf $\mathcal K$ is frequently called the {\bfseries kernel} of the corresponding functor.
\end{Cor}

\begin{Ex}\label{examples_of_kernels}
Let $\xymatrix{X \ar[r]^-{f} & X}$ be an endomorphism of a perfect stack. We claim that the kernel of the functor $\xymatrix{\QCoh(X) \ar[r]^-{f_*} & \QCoh(X)}$ is given by the pushforward $(\Gamma_f)_* \mcal{O}_X \in \QCoh(X \times X)$, where $\Gamma_f\colon X \to X\times X$ is the graph of $f$. Indeed, for a sheaf $\mathcal F\in \QCoh(X)$ we have a chain of natural equivalences
$${q_2}_* \big{(} (\Gamma_f)_* \mcal{O}_X  \otimes q_1^* \mcal{F} \big{)} \simeq {q_2}_* (\Gamma_f)_* \big{(} \mcal{O}_X \otimes (\Gamma_f)^* q_1^* \mcal{F} \big{)} \simeq f_* \mcal{F}
$$
where the first equivalence is the projection formula (see \cite[Proposition 3.10]{BFN}). As a special case, we see that the kernel of the identity functor on $\QCoh(X)$ is given by $\Delta_*\mathcal O_X \in \QCoh(X \times X)$, where $\xymatrix{X \ar[r]^-{\Delta} & X\times X}$ is the diagonal morphism.
\end{Ex}

\medskip

In particular, for a functor $F \in \Hom_{\Cat_k}(\QCoh(X),\QCoh(X))$ where $X$ is a perfect derived stack it makes sense to calculate the trace $\Tr_{\Cat_k}(F)$ of $F$ in terms of the corresponding kernel $\mcal{K}_F \in \QCoh(X \times X)$. Namely, we have

\begin{Prop}\label{trace_from_kernel} Let $F \in \Hom_{\Cat_k}(\QCoh(X),\QCoh(X))$ be a functor where $X$ is a perfect derived stack with the kernel given by $\mcal{K}_F \in \QCoh(X \times X)$. Then
$$
\Tr_{\Cat_k}(F) \simeq \Gamma(X,\Delta^* \mcal{K}_F) \simeq \Gamma(X \times X,\mcal{K}_F \otimes \Delta_* \mcal{O}_X) \in \Vect_k.
$$

\begin{proof}
By definition the trace is the composition 
$$\xymatrix{
\Vect_k \ar[rr]^-{\Delta_* \circ p^*} && \QCoh(X \times X) \ar[rr]^-{F \otimes \Id_{\QCoh(X)}} && \QCoh(X \times X) \ar[rr]^-{p_* \circ \Delta^*} && \Vect_k
}$$
which due to the following commutative diagram
$$\xymatrix{
\QCoh(X\times X)\ar[rr]\ar[d]_\sim && \QCoh(X\times X) \ar[d]^\sim\\
\QCoh(X)\otimes \QCoh(X) \ar[rr]^-{F\otimes \Id_{\QCoh(X)}}\ar[d]_\sim && \QCoh(X)\otimes \QCoh(X) \ar[d]^\sim\\
\Fun^{L}(\QCoh(X), \QCoh(X)) \ar[rr]_-{F \circ -} && \Fun^{L}(\QCoh(X), \QCoh(X)) \\
}$$
is given by
$$\xymatrix{
k \ar@{|->}[rr] && \Delta_* \mcal{O}_X \ar@{|->}[rr] && \mcal{K}_F \ar@{|->}[rr] && \Gamma(X,\Delta^* \mcal{K}_F)
}$$
so that we instantly obtain the desired equivalence.
\end{proof}
\end{Prop}

\begin{Ex}\ \label{ex:tr_for_qcoh}
As a toy example, consider the case where $X=\Spec R$ for $R \in \CAlg(\Vect_k)$ is affine and $F \in  \Hom_{\Cat_k}(\Mmod_R,\Mmod_R)$. Then the kernel of the functor $F$ can be described explicitly as $F(R) \in \Mmod(R \otimes R)$, where the bimodule structure arises from the fact that $R$ is commutative. Consequently, we get an equivalence
$$
\Tr_{\Cat_k}(F) \simeq \Gamma(X, \Delta^*F(R)) \simeq F(R)\otimes_{R\otimes R} R.
$$
\end{Ex}

\subsection{Calculating the trace}
\begin{Def} \label{def:derived_intersection_stack}
Let $f$ be an endomorphism of a derived stack $X$. Define a \emph{derived fixed point stack} $X^f$ of $f$ as the pullback
$$\xymatrix{
X^f \ar[r]^-{i} \ar[d]_-{i} & X \ar[d]^-{\Gamma_f} 
\\
X \ar[r]_-{{\Delta}} & X \times X
}$$
in the $\ii$-category of derived stacks. Note that by \cite[Proposition 3.24]{BFN} in the case when $X$ is perfect so is the derived fixed points stack $X^f$.
\end{Def}

Later on we will need the following

\begin{Prop} \label{prop:tr_of_f}
In the case when $X$ is perfect we have
$$
\Tr_{\Cat_k}(f_*) \simeq \Gamma(X^f, \mathcal O_{X^f}).
$$
\begin{proof}
 Since by \ref{examples_of_kernels} the kernel of the functor $f_*$ is given by $(\Gamma_f)_* \mcal{O}_X$ we have
$$\Tr_{\Cat_k}(f_*) \simeq \Gamma(X,\Delta^* (\Gamma_f)_* \mcal{O}_X).$$
Considering the pullback diagram
$$\xymatrix{
X^f \ar[r]^-{i} \ar[d]_-{i} & X \ar[d]^-{\Gamma_f} 
\\
X \ar[r]_-{{\Delta}} & X \times X
}$$
and using the projection formula we obtain an equivalence
$$
\Gamma(X,\Delta^* (\Gamma_f)_* \mcal{O}_X)  \simeq \Gamma(X,i_*i^*\mcal{O}_X) \simeq \Gamma(X,i_*\mcal{O}_{X^f}) \simeq \Gamma(X^f, \mathcal O_{X^f})$$
as claimed.
\end{proof}
\end{Prop}

\medskip

Let now $X$ be perfect derived stack together with an endomorphism $\xymatrix{X \ar[r]^-{f} & X}$ and let $\xymatrix{X^f \ar@{^(->}[r]^i & X}$ be the inclusion of the derived fixed points stack. Suppose we are given a diagram
$$\xymatrix{
\Vect_k \ar@/_/[dd]_-{\varphi} \ar[rr]^-{\Id_{\Vect_k}} && \Vect_k \ar@/_/[dd]_-{\varphi} \ar@2[ddll]_-{T}
\\
\\
\QCoh(X) \ar@/_/[uu]_-{\psi} \ar[rr]_-{f_*} && \QCoh(X) \ar@/_/[uu]_-{\psi}
}$$
in $2\Cat_k$, where $\varphi$ is the left adjoint $\psi$. Set $E:=\varphi(k)$ so that the $2$-morphism $T$ classifies some morphism $t \in \Hom_{\QCoh(X)}(E,f_*E)$. 

Note that by assumption on $X$ a quasi-coherent sheaf $E \in \QCoh(X)$ is compact if and only if it is dualizable. Indeed, in our case we have $\QCoh(X) \simeq \Ind \Perf(X)$ and $\QCoh(X)^{\dual} \simeq \Perf(X)$ and since by definition perfect complexes are closed under retracts we get an equivalence $\QCoh(X)^{\omega} \simeq \Perf(X)\simeq \QCoh(X)^{\dual}$. In particular, in what follows we can freely switch between the notions of compact and dualizable objects.

We now can calculate the Chern character $\ch(E,t)$ (definition \ref{def:chern}) in terms of $E$. By the previous proposition $\ch(E,t)$ can be seen as a map $\xymatrix{k \ar[r] & \Gamma(X^f,\mcal{O}_{X^f})}$ classifying some element of $\Gamma(X^f,\mcal{O}_{X^f}) \simeq \Hom_{\QCoh(X^f)}(\mathcal O_{X^f}, \mathcal O_{X^f})$. The next proposition provides a more concrete description of this endomorphism:

\medskip

\begin{Prop}\label{prop:chern_in_ag}
We have
$$\ch(E,t) = \Tr_{\QCoh(X^f)}\left(\xymatrix{i^* E \simeq i^* f^* E \ar[r]^-{i^*(b)} & i^* E}\right)$$
where $\xymatrix{f^*E \ar[r]^-{b} & E}$ is the morphism which corresponds to $t \in \Hom_{\QCoh(X)}(E,f_* E)$ using the adjunction between $f^*$ and $f_*$.

\begin{proof}
By definition \ref{def:chern} the Chern character $\ch(E,t)$ is the composite in $\Vect_k$
$$\xymatrix{
k \ar[r]^-{(1)} & \Gamma(X, E \otimes E^{\vee}) \ar[r]^-{(2)} & \Gamma(X,f_*E \otimes E^{\vee}) \ar[r]^-{(3)} & \Gamma(X, \Delta^* (\Gamma_f)_* \mcal{O}_X)
}$$
obtained from the $2$-commutative diagram
$$\xymatrix{
\Vect_k \ar[dd]_-{\Id_{\Vect_k}} \ar[rr]^-{\Id_{\Vect_k}} && \Vect_k \ar@2[ddll]  \ar[rr]^-{\Id_{\Vect_k}} \ar[dd]_-{E \boxtimes E^{\vee}} && \Vect_k \ar[rr]^-{\Id_{\Vect_k}} \ar[dd]^-{E \boxtimes E^{\vee}} \ar@2[ddll] && \Vect_k \ar[dd]^-{\Id_{\Vect_k}} \ar@2[ddll]
\\
\\
\Vect_k \ar[rr]_-{\Delta_* \mcal{O}_X} && \QCoh(X \times X) \ar[rr]_-{(f \times \Id_X)_*} && \QCoh(X \times X) \ar[rr]_-{\Gamma(X,\Delta^* -)} && \Vect_k
}$$
where $E \boxtimes E^{\vee}:=q_1^* E \otimes q_2^* E^{\vee}$ and
$$\xymatrix{
X && X \times X \ar[ll]_-{q_1} \ar[rr]^-{q_2} && X
}$$
are the projection maps. We first note that the composition $(2) \circ (1)$ is induced by the choice of the morphism 
$$
t \in \Hom_{\QCoh(X)}(E,f_*E) \simeq \Gamma\big{(}X,f_*E \otimes E^{\vee}\big{)}.
$$
Now the morphism $(3)$ is obtained by applying the functor $\Gamma\big{(}X, \Delta^* (f \times \Id_X)_* -\big{)}$ to the composite
$$\xymatrix{
\psi: E \boxtimes E^{\vee} \ar[r] & \Delta_* \Delta^* (E \boxtimes E^{\vee}) \simeq \Delta_* (E \otimes E^{\vee}) \ar[rr]^-{\Delta_*(\ev_{E})} && \Delta_* \mcal{O}_X
}$$
where the first map is induced by the unit of the adjunction between $\Delta^*$ and $\Delta_*$. Using the pullback square
$$\xymatrix{
X \ar[r]^-{\Gamma_f} \ar[d]_-{f} & X \times X \ar[d]^-{f \times \Id_X} 
\\
X \ar[r]_-{\Delta} & X \times X
}$$
and the corresponding base change we obtain an equivalence of functors
\begin{align}\label{eq1}
\Gamma\big{(}X, \Delta^* (f \times \Id_X)_* - \big{)} \simeq \Gamma\big{(}X, f_* (\Gamma_f)^* - \big{)}
\end{align}
and so the morphism $\Gamma\big{(}X, \Delta^* (f \times \Id_X)_* \psi\big{)}$ can be obtained by applying $\Gamma(X,f_*-)$ to the composite
$$\xymatrix{
(\Gamma_f)^*( E \boxtimes E^{\vee}) \ar[r] & (\Gamma_f)^*\Delta_*(E \otimes E^{\vee}) \ar[r] & (\Gamma_f)^* \Delta_* \mcal{O}_X.
}$$
Now applying the equivalence of functors $(\Gamma_f)^*\Delta_* \simeq i_* i^*$ obtained from the pullback square
$$\xymatrix{
X^f \ar[r]^-{i} \ar[d]_-{i} & X \ar[d]^-{\Gamma_f} 
\\
X \ar[r]_-{{\Delta}} & X \times X
}$$
we can rewrite the latter morphism as
$$\xymatrix{
E \otimes f^* E^{\vee}  \ar[r] & i_* i^*(E \otimes f^* E^{\vee} ) \simeq i_* i^* (E \otimes E^{\vee}) \ar[r] & i_* \mcal{O}_{X^g}
}$$
where the first morphism is given by the unit of the adjunction $i_* \dashv i^*$ and the second morphism is induced by $\ev_{E}$. Consequently, we see that the whole Chern character can we written as the composite
$$\xymatrix{
\mcal{O}_{X} \ar[r] & E \otimes E^{\vee}\ar[rr]^-{t \otimes \Id_{E^{\vee}}} && f_*E \otimes E^{\vee} \simeq f_*(E \otimes f^*E) \ar[r] & f_*i_*i^* (E \otimes E^{\vee}) \ar[r] & f_*i_* \mcal{O}_{X^g}.
}$$
Now to finish the proof it is left to use the equivalence $f_* i_* \simeq i_*$, the adjunction $i^* \dashv i_*$ and that the morphism $b \in \Hom_{\QCoh(Y)}(f^*E, E)$ is defined to correspond via the adjunction $f^* \dashv f_*$ to the morphism $t \in \Hom_{\QCoh(X)}(E,f_* E)$.
\end{proof}
\end{Prop}

\begin{Rem}
More generally in the case of a correspondence $\xymatrix{X & Y \ar[l]_-{g} \ar[r]^-{f} & X}$ between perfect stack calculation similar to \ref{prop:tr_of_f} proves that $\Tr_{\Cat_k}(g^*f_*) \simeq \Gamma(X^{(g,f)}, \mathcal O_{X^{(g,f)}})$ where $X^{(g,f)}$ is the pullback
$$\xymatrix{
X^{(g,f)} \ar[r] \ar[d] & Y \ar[d]^-{(g,f)}
\\
X \ar[r]_-{\Delta} & X \times X
}$$
Given a dualizable sheaf $E \in \QCoh(X)$ together with a morphism $\xymatrix{E \ar[r]^-{t} & f_* g^*E}$ the proof of \ref{prop:chern_in_ag} can be easily adapted to show that there is an equality
$$\ch(E,t) = \Tr_{\QCoh(X^{(g,f)})}\left(\xymatrix{i^* E \simeq j^* f^* E \ar[r]^-{j^*(b)} & j^*g^* E \simeq i^*E}\right)$$
where $\xymatrix{f^*E \ar[r]^-{b} & g^* E}$ is the morphism which corresponds to $t \in \Hom_{\QCoh(X)}(E,f_*g^* E)$ using the adjunction between $f^*$ and $f_*$.
\end{Rem}

\section{Holomorphic Atiyah-Bott formula}

\begin{Conv} \label{AB_convs} For the rest of this section we will work in the following setting:
\\
1) $k$ is an algebraically closed base field.
\\
2) $X$ is a smooth proper variety over $k$ together with an endomorphism $f$ such that its graph 
$\xymatrix{X \ar[r]^-{\Gamma_f} & X \times X}$ intersects the diagonal $\xymatrix{X \ar[r]^-{\Delta} & X \times X}$ transversally.

Note that in this case we have
$$\left(\mathcal O_\Delta \otimes_{\mathcal O_{X\times X}} \mathcal O_{\Gamma_f}\right)_x \simeq \mathcal O_{\Delta,x} \otimes_{\mathcal O_{X\times X,x}} \mathcal O_{\Gamma_f,x} \simeq k$$
as $\Delta$ and $\Gamma_f$ are complete intersections by the set of transversal sections of the (stalk of) the tangent bundle $T_{X\times X,x}$. It follows that the derived fixed point scheme $X^f$ equals the ordinary intersection $\Delta \cap_{X\times X} \Gamma_f$ and is reduced. As it is proper and of dimension zero, we conclude that $X^f$ is just a disjoint union of finitely many points. We will further denote by $p$ the projection map $\xymatrix{X \ar[r]^-{p} & \ast}$.
\\
\\
3) $E \in \QCoh(X)$ is a dualizable and {\bfseries lax equivariant} quasi-coherent sheaf over $X$, that is, there is a fixed morphism 
$$\xymatrix{
f^* E \ar[r]^-{b} & E.
}$$ 
in $\QCoh(X)$. Using the adjunction between $f^*$ and $f_*$ we will denote by $t \in \Hom_{\QCoh(X)}(E,f_* E)$ the morphism which corresponds to $b$.
\end{Conv}

\subsection{Statement of Atiyah-Bott formula}

We begin with the following

\begin{Def} \label{def:Lef}
Define a {\bfseries Lefschetz number} $\Ll(E, b) \in k$  of $b$ as the trace
$$\xymatrix{
\Ll(E, b):=\Tr_{\Vect_k} \Big(\Gamma(X, E) \ar[r] & \Gamma(X,f_*f^*E) \simeq \Gamma(X, f^*E) \ar[r]^-{\Gamma(b)} & \Gamma(X, E)\Big).
}$$
\end{Def}

Our main goal is to use the formalism of traces in $(\infty,2)$-categories discussed above in order to prove the 
   
\begin{Theor}[Holomorphic Atiyah-Bott formula] \label{thm:atiyah_bott}
We have
\begin{align} \label{frml:atiyah_bott}
\Ll(E, b) = \sum_{x=f(x)} \frac{\Tr_{\Vect_k}(E_x \simeq E_{f(x)} \stackrel{b_x}{\longrightarrow} E_x)}{\det(1-d_xf)}
\end{align}
where $\xymatrix{T_{X,x} \ar[r]^-{d_x f } & T_{X,x}}$ is the differential of $f$ from the tangent space at a point $x \in X$ to itself (by basic linear algebra transversality of $\Delta$ and $\Gamma_f$ in $X\times X$ is equivalent to invertibility of $1-d_x f$ for all fixed points $x$, so denominators on the right hand side of (\ref{frml:atiyah_bott}) are nonzero and the formula makes sense).
\end{Theor}

\begin{Ex}
Take $X:=\mathbb P^1$ with homogeneous coordinates $(z:w)$. Let $f(z)=e^{i\phi}z$ be a rotation automorphism by some nonzero angle $\phi$ and set $E:=\mathcal O_{\mathbb P^1}(n)$. Note that $\mathcal O_{\mathbb P^1}(-1)$ has a tautological $\GL_2$-equivariant structure which induces a $\GL_2$-equivariant structure on $\mathcal O_{\mathbb P^1}(n)$. We then define an $f$-equivariant structure on $E$ by considering $e^{i\phi}$ as an element of $\GL_2$
$$\left(\begin{matrix}
e^{i\phi} & 0 \\
0 &         1
\end{matrix}\right) \in \GL_2 .$$

The morphism $f$ has two fixed points $0$ and $\infty$. The stalks of $E$ at $0$ and $\infty$ are generated by $w^n$ and $z^n$ respectively. Hence the right hand side of Atiyah-Bott in this case is equal to
\begin{align}
 \label{ex:P1_RHS}
\frac{1}{1-e^{i\phi}} + \frac{e^{i\phi n}}{1-e^{-i\phi}}=\frac{e^{i\phi(n+1)}-1}{e^{i\phi}-1}
\end{align}

For the left hand side of Atiyah-Bott we have three slightly different cases

\begin{itemize}
\item Let $n\ge 0$. In this case $H^0(\mathbb P^1, \mathcal O_{\mathbb P^1}(n))$ is the only nontrivial cohomology group of $\Gamma(\mathbb P^1, \mathcal O_{\mathbb P^1}(n))$ with the basis of the form $z^k w^{n-k}, 0\le k \le n$. It follows that the Lefschetz number $L$ is equal to
$$
L=\Tr_{\Vect_k^{\heartsuit}}(f^*_{|H^0(\mathbb P^1, \mathcal O_{\mathbb P^1}(n))}) = \sum_{k=0}^n e^{i\phi k} = \frac{e^{i\phi(n+1)}-1}{e^{i\phi}-1}
$$
which coincides with (\ref{ex:P1_RHS}).

\item Let $n<-1$. By Serre duality the only nontrivial cohomology group of $\Gamma(\mathbb P^1, \mathcal O_{\mathbb P^1}(n))$ is
$$H^1(\mathbb P^1, \mathcal O_{\mathbb P^1}(n))\simeq H^0(\mathbb P^1, O_{\mathbb P^1}(-n-2))^\vee$$
with the action $z^k w^{-n-2-k} \mapsto e^{-i\phi(k+1)}z^k w^{n-k}$, $k=0,\ldots, -n-2$. We then have
$$L = -\Tr_{\Vect_k^{\heartsuit}}(f^*_{|H^1(\mathbb P^1, \mathcal O_{\mathbb P^1}(n))})= -\sum_{k=0}^{-n-2} e^{-i\phi (k+1)}=-e^{-i\phi}\frac{1-e^{i\phi(-n-1)}}{1-e^{-i\phi}}=\frac{e^{i\phi(n+1)}-1}{e^{i\phi}-1}$$
which again coincides with (\ref{ex:P1_RHS}).

\item Let $n=-1$. The sheaf $\mathcal O_{\mathbb P^1}(-1)$ is acyclic and both sides of Atiyah-Bott are equal to zero.
\end{itemize}

\end{Ex}

\begin{Cor}
We have
$$
\sum_{q=0}^{\dim X} (-1)^q \Tr_{\Vect_k^{\heartsuit}}(f^*_{|H^q(X, \mathcal O_X)}) = \sum_{x=f(x)} \frac{1}{\det(1-d_x f)}.
$$

\begin{proof}
Just set $E:=\mathcal O_X$ and $b:=\Id_{\mathcal O_X}$.
\end{proof}
\end{Cor}

\begin{Cor}[Weyl character formula]
Let $G$ be a semisimple simply connected Lie group over $\mathbb C$, and $T, B$ be the maximal torus and the Borel subgroup of $G$ respectively. Take $V=V(\lambda)$ an irreducible finite dimensional representation of $G$ with highest weight $\lambda$ and let $\chi_V$ be its character. Then
$$\chi_V = \frac{\sum_{w\in W} \epsilon(w) e^{w(\lambda + \rho)}}{e^{\rho}\prod_{\alpha\in \Delta^+}(1-e^{-\alpha})}$$
where
\begin{itemize}
\item $W$ is the Weyl group,

\item $\Delta^+$ is subset of positive roots of root system $\Delta$,

\item $\rho$ is the half sum of positive roots,

\item $\epsilon(w)=(-1)^{l(w)}$, where $l(w)$ is the length of the Weyl group element, defined to be the minimal number of reflections with respect to the simple roots such that $w$ equals the product of those reflections.
\end{itemize}

\begin{proof}
Let $X:=G/B$ be a flag variety which is proper since $B$ is Borel subgroup. Set $\mathcal L_\lambda := G \times_B \mathbb C(\lambda)$ a $G$-equivariant line bundle over $X$. By the Borel-Weil-Bott theorem we have
$$H^i(X, \mathcal L_\lambda) \simeq \left \{ \begin{matrix}
V, \quad i=0 \\
0, \quad i\ne 0.
\end{matrix} \right.$$

Let $t\in T$ be generic and denote by $\xymatrix{X \ar[r]^-{L_t} & X}$ the action by left translation by $t$. Then the tangent space $T_{wB} X$ to a fixed point $wB \in X^t$ is isomorphic (as $T$-representation) to $\mathfrak g/w(\mathfrak b)$ and hence
$$
\det(1-d_{wB}L_t) = \prod_{\alpha\in \Delta^+} (1-e^{-w\alpha}(t)).
$$

By definition of $\mathcal L_\lambda$ the action of $L_t$ on the fiber $(\mathcal L_\lambda)_{wB}$ at a fixed point $wB$ is given by the multiplication by $e^{w\lambda}(t)$.

For a regular element $t\in T$ the graph of $L_t$ intersects diagonal transversely so that we can apply Atiyah-Bott to get
\begin{align*}
\chi_V(t) & = \Tr_{\Vect_k^{\heartsuit}}(L_{t|H^0(X, \mathcal L_\lambda)})= \sum_{w\in W} \frac{e^{w\lambda}(t)}{\prod_{\alpha\in \Delta^+} (1-e^{-w\alpha}(t))}=\\
& = \sum_{w\in W} \frac{e^{w\lambda}(t)}{\epsilon(w) e^{-w\rho}(t)\prod_{\alpha\in \Delta^+} (e^{\alpha/2}(t)-e^{-\alpha/2}(t))} = \frac{\sum_{w\in W} \epsilon(w) e^{w(\lambda + \rho)}(t)}{e^{\rho}(t)\prod_{\alpha\in \Delta^+}(1-e^{-\alpha}(t))}.
\end{align*}
Now the result in general case follows from the fact that regular elements are dense in $T$.
\end{proof}
\end{Cor}

\subsection{Proof of Atiyah-Bott formula}
We will prove the theorem \ref{thm:atiyah_bott} by interpreting both sides of the Atiyah-Bott formula as morphisms between certain traces.

\begin{Plan} \label{plan:atiyah_boott} First note that for a smooth proper $X$ the functor $\Gamma=p_*$ admits a right adjoint $p^! \simeq p^* \otimes \omega_X$ which is also continuous since both $p^*$ and $-\otimes \omega_X$ are. Hence by \ref{rem:dual_arrows_in2cat} the arrow $\xymatrix{\QCoh(X) \ar[r]^-\Gamma & \Vect_k}$ is an object of $\Arr(\Cat_k)^{\dual}$ so it makes sense to speak about traces of its endomorphisms. Now applying the functoriality of traces \ref{prop:functoriality_of_traces} to the diagram
$$\xymatrix{
{\Vect_k} \ar[d]_-{E} \ar[rr]^-{\Id_{{\Vect_k}}} & & {\Vect_k} \ar[d]^-{E} \ar@2[dll]_-{T}
\\
\QCoh(X) \ar[d]_{\Gamma} \ar[rr]_-{f_*} && \QCoh(X)  \ar[d]^-{\Gamma} 
\\
{\Vect_k} \ar[rr]_-{\Id_{{\Vect_k}}} && {\Vect_k}
}$$
in $2\Cat_k$ we obtain a commutative triangle
$$\xymatrix{
\Tr_{2\Cat_k}(\Id_{{\Vect_k}}) \ar[rr]^-{\ch(E,t)} \ar[drr]_-{\ch(\Gamma(X,E),\Id_{\Gamma(X,E)} \circ t) \indent \indent} && \Tr_{2\Cat_k}(f_*) \ar[d]^-{\int_{X^f}} .
\\
&& \Tr_{2\Cat_k}(\Id_{{\Vect_k}})
}$$
in $\Vect_k$, where $\int_{X^f} := \Tr(\Gamma, \Id_{\Gamma})$. Since $\Tr_{2\Cat_k}(\Id_{{\Vect_k}}) \simeq k$ we get an equality 
$$\ch(\Gamma(X,E),\Id_{\Gamma(X,E)}\circ t)=\int_{X^f} \ch(E,t)$$
 of two {\bfseries numbers}. The desired proof will now follow from the calculation that the right-hand and the left-hand sides of this equality are the right-hand and the left-hand sides of the Atiyah-Bott formula respectively.
\end{Plan}

Notice that we instantly get the following

\begin{Lemma} \label{prop:chp_eq_lefz}
We have
$$\ch(\Gamma(X,E),\Id_{\Gamma(X,E)}\circ t)=\Ll(E, b). \qed$$
\end{Lemma}

\medskip

We now wish to calculate the morphism 

$$\xymatrix{
k \simeq\Tr_{2\Cat_k}(\Id_{\Vect_k}) \ar[rr]^-{\ch(E,t)} && \Tr_{2\Cat_k}(f_*).
}$$ 
By the assumptions on $X$ and $f$ this can be done instantly:

\begin{Prop}
We have
$$\Tr_{2\Cat_k}(f_*) \simeq \bigoplus_{x=f(x)} ke_x$$
where for each fixed point $x$ we set $e_x:=1\in \Gamma(\{x\}, \mathcal O_{x})$.

\begin{proof}
Follows from proposition \ref{prop:tr_of_f} because for $X$ and $f$ in our setting \ref{AB_convs} the derived intersection $X^f$ is equivalent to the ordinary (discrete) one. 
\end{proof}
\end{Prop}

It follows that we can write
$$\ch(E,t) = \sum\limits_{x=f(x)} \ch(E,t)_x e_x$$
for some $\ch(E,t)_x\in k$. 

\begin{Prop} \label{prop:calculatin_ch}
We have
$$\xymatrix{
\ch(E,t)_x = \Tr_{\Vect_k}(E_x\simeq E_{f(x)} \ar[r]^-{b_x} & E_x)
}$$
where by $E_x$ we mean the derived stalk of $E$ at the point $x \in X$.
\begin{proof}
Instant from proposition \ref{prop:chern_in_ag}.
\end{proof}
\end{Prop}

\medskip

It is only left now to understand the functional
$$\xymatrix{
\Tr_{2\Cat_k}(f_*) \ar[rr]^-{\int_{X^f}} && \Tr_{2\Cat_k}(\Id_{{\Vect_k}}) \simeq k.
}$$
We have the following

\begin{Prop} \label{prop:tr_of_gamma}
The map
$$\xymatrix{
\bigoplus\limits_{x=f(x)} k e_x \simeq \Tr_{2\Cat_k}(f_*) \ar[rr]^-{\int_{X^f}} && \Tr_{2\Cat_k}(\Id_{{\Vect_k}}) \simeq k
}$$
sends $e_x$ to $1/{\det(1-d_x f)} \in k$, where $\xymatrix{T_{X,x} \ar[r]^-{d_x f } & T_{X,x}}$ is the differential of $f$ from the tangent space at the point $x \in X$ to itself.
\end{Prop}

From proposition \ref{prop:tr_of_gamma} it is straightforward to derive the 
\begin{proof}[Proof of Atiyah-Bott formula~\ref{thm:atiyah_bott}]
From proposition \ref{prop:functoriality_of_traces} we know that
$$\ch(\Gamma(X,E),\Id_{\Gamma(X,E)}\circ t) = \int_{X^f} \ch(E,t).$$
But by proposition \ref{prop:chp_eq_lefz} we have
$$\ch(\Gamma(X,E),\Id_{\Gamma(X,E)}\circ t)=\Ll(E, b)$$
and by propositions \ref{prop:calculatin_ch} and \ref{prop:tr_of_gamma} we have
$$\int_{X^f} \ch(E,t)= \sum_{x=f(x)} \ch(E,t)_x \cdot \int_{X^f}(e_x) = \sum_{x=f(x)} \frac{\Tr_{\Vect_k}(E_x \stackrel{b_x}{\longrightarrow} E_x)}{\det(1-d_x f)}.$$
\end{proof}

\begin{proof}[Proof of Proposition~\ref{prop:tr_of_gamma}]
Set $\lambda_x:=\int_{X^f}(e_x)$. In order to to find $\lambda_x$ we will apply the formula
$$\ch(\Gamma(X,E),\Id_{\Gamma(X,E)}\circ t)=\int_{X^f} \ch(E,t)$$
in the simplest case $E:=x_* \mcal{O}_k,$ the skyscraper sheaf of a fixed point $x \in X$ which we consider as a lax equivariant sheaf with $t:=\Id_{x_* \mcal{O}_k}$. In this case we have an equality
$$\xymatrix{
1=\Ll(E, b)=\int_{X^f} \ch(\mcal{O}_x,\Id_{x_* \mcal{O}_k})=\lambda_x \Tr_{\Vect_k} \Big( (x_* \mcal{O}_k)_x \ar[r]^-{b_x} &(x_* \mcal{O}_k)_x \Big)
}$$
so that it is left to calculate the trace $\xymatrix{\Tr_{\Vect_k}  \Big( (x_* \mcal{O}_k)_x \ar[r]^-{b_x} &(x_* \mcal{O}_k)_x \Big)}$. Since by the assumption our variety $X$ is smooth at $x$, we have $H^p(X, x_* \mcal{O}_k) \simeq \Lambda^p(T_{X,x})$. Now the statement follows by setting $V:=T_{X,x}$, $A:=d_x f$ and $s=-1$ in the following well-known linear algebra lemma.

\begin{Lemma}
Let $\xymatrix{V \ar[r]^-{A} & V}$ be a linear map from a finite dimensional vector space $V$ to itself, and set $p_A(s) := \det(1 + sA)$. Then
$$p_A(s) = \sum_{p=0}^{\dim V} \Tr \Lambda^p(A) s^p. $$
\end{Lemma}
\end{proof}

\bigskip
\bigskip

\noindent
Grigory~Kondyrev, {\sc Northwestern University; National Research University Higher School of Economics, Russian Federation,}
\href{mailto:gkond@math.northwestern.edu}{gkond@math.northwestern.edu}

\smallskip

\noindent
Artem~Prikhodko, {\sc National Research University Higher School of Economics, Russian Federation; Center for Advanced Studies, Skoltech, Russian Federation,}
\href{mailto:artem.n.prihodko@gmail.com}{artem.n.prihodko@gmail.com}

\end{document}